\documentclass[draft]{amsart}
\usepackage{mathrsfs}
\usepackage{amssymb}
\usepackage{amsmath}
\usepackage{amsfonts}
\usepackage{amsfonts,enumerate}
\usepackage{pifont}
\usepackage{enumerate}
\usepackage{graphics}
\usepackage{verbatim}
\usepackage{amsthm}

\newtheorem{thm}{Theorem}[section]
\newtheorem{prop}[thm]{Proposition}

\newtheorem{lem}[thm]{Lemma}
\newtheorem{defi}[thm]{Definition}
\newtheorem{remark}[thm]{Remark}
\newtheorem{example}[thm]{Example}
\newtheorem{pb}[thm]{Problem}

\newenvironment{rk}{\begin{remark}\rm}{\end{remark}}

\newenvironment{ex}{\begin{example}\rm}{\end{example}}

\numberwithin{equation}{section}

\newcommand{\8}{\infty}

\newcommand{\be}{\begin{eqnarray*}}
\newcommand{\ee}{\end{eqnarray*}}
\newcommand{\beq}{\begin{equation}}
\newcommand{\eeq}{\end{equation}}
\newcommand{\beqn}{\begin{equation*}}
\newcommand{\eeqn}{\end{equation*}}

\begin{document}

\title{Area integral functions and $H^{\8}$ functional calculus for sectorial operators on Hilbert spaces}


\author{Zeqian Chen}

\address{Wuhan Institute of Physics and Mathematics, Chinese Academy of Sciences, 30 West District, Xiao-Hong-Shan, Wuhan 430071,China}
\email{zqchen@wipm.ac.cn}

\thanks{Z. Chen is partially supported by NSFC grant No. 11171338.}

\author{Mu Sun}

\address{Wuhan Institute of Physics and Mathematics, Chinese Academy of Sciences, 30 West District, Xiao-Hong-Shan, Wuhan 430071, China
and Graduate University of Chinese Academy of Sciences, Beijing 100049, China}


\date{}
\maketitle

\markboth{Z. Chen and M. Sun}%
{$H^{\8}$ functional calculus}

\begin{abstract}
Area integral functions are introduced for sectorial operators on Hilbert spaces. We establish the equivalence relationship between the square and area integral functions. This immediately extends McIntosh/Yagi's results on $H^{\8}$ functional calculus of sectorial operators on Hilbert spaces to the case when the square functions are replaced by the area integral functions.
\end{abstract}


\section{preliminaries}\label{pre}

The theory of sectorial operators, their $H^{\8}$ functional calculus, and their associated square functions on Hilbert spaces grew out from McIntosh's seminal paper \cite{M1986} and a subsequent work by McIntosh/Yagi \cite{MY1989}, and then was generalized to the setting of Banach spaces by Cowling-Doust-McIntosh-Yagi \cite{CDMY1996} and by Kalton/Weis \cite{KW2001}. The aim of this paper is to introduce so-called area integral functions for sectorial operators on Hilbert spaces and to extend McIntosh/Yagi's theory to the case when the square functions are replaced by the area integral functions. The corresponding $L_p$ case will be given elsewhere \cite{CS}.

To this end, in this section we give a brief review of $H^{\8}$ functional calculus on general Banach spaces, and preliminary results that will be used for what follows. We mainly follow the fundamental works \cite{CDMY1996, M1986, MY1989}. See also \cite{ADM1996, LeM2007} for further details. We refer to \cite{Gold1985} for the necessary background on semigroup theory.

\subsection{Sectorial operators and $C_0$-semigroups}

Let $\mathbf{X}$ be a complex Banach space. We denote by $\mathcal{B} ( \mathbf{X})$ the Banach algebra of all bounded operators on $\mathbf{X}.$ Let $A$ be a closed and densely defined operator on $\mathbf{X}.$ We let $\mathrm{D} (A),$ $\mathrm{N} (A)$ and $\mathrm{R} (A)$ denote the domain, kernel and range of $A$ respectively. Further we let $\sigma (A)$ and $\rho (A)$ denote the spectrum and resolvent set of $A$ respectively. Then, for any $\lambda \in \rho(A),$ we let
\be
R(\lambda, A)=(\lambda-A)^{-1}
\ee
denote the corresponding resolvent operator.

For any
$\omega\in(0,\pi)$, we let
\be
\Sigma_\omega = \{z \in \mathbb{C}\setminus \{0\} : \; | \mathrm{Arg} (z) | < \omega\}
\ee
be the open sector of angle $2\omega$ around the half-line $(0, \infty).$ Then, $A$ is said to be a sectorial operator of type $\omega$ if $A$ is closed and densely defined, $\sigma (A) \subset \overline{\Sigma}_\omega,$ and for any $\theta \in(\omega,\pi)$ there is a constant $K_\theta>0$ such that
\beq\label{eq:EsitSectorialOper}
| z R(z,A) | \le K_\theta, \quad z \in \mathbb{C} \setminus \overline{\Sigma}_\theta .
\eeq
We say that $A$ is sectorial of type $0$ if it is of type $\omega$ for any $\omega >0.$

Let $(T_t)_{t \ge 0}$ be a bounded $C_0$-semigroup on $\mathbf{X}$ and let $- A$ denote its infinitesimal generator. Then $A$ is closed and densely defined. Moreover, $\sigma (A) \subset \overline{\Sigma}_{\frac{\pi}{2}}$ and, for any $\lambda \in \mathbb{C} \setminus \overline{\Sigma}_{\frac{\pi}{2}}$ we have
\be
R (\lambda, A ) = - \int^{\8}_0 e^{\lambda t} T_t d t
\ee
in the strong operator topology, from which it follows that $A$ is a sectorial operator of type $\frac{\pi}{2}.$

\begin{prop}\label{prop:AnalyticSemigroup}{\rm ( see e.g. \cite{Gold1985})}
Let $(T_t)_{t \ge 0}$ be a bounded $C_0$-semigroup on $\mathbf{X}$ with the infinitesimal generator $- A.$ Given $\omega \in (0, \frac{\pi}{2}),$ the following are equivalent:
\begin{enumerate}[{\rm (i)}]

\item $A$ is sectorial of type $\omega.$

\item For any $\alpha \in (0, \frac{\pi}{2} - \omega ),$ $(T_t)_{t \ge 0}$ admits a bounded analytic extension $(T_z)_{z \in \Sigma_{\alpha}}$ in $\mathcal{B} ( \mathbf{X}).$

\end{enumerate}
\end{prop}

By definition, a $C_0$-semigroup $(T_t)_{t \ge 0}$ is called a bounded analytic semigroup if there exist a positive angle $0 < \alpha <\frac{\pi}{2}$ and a bounded analytic extension of $(T_t)_{t \ge 0}$ on $\Sigma_\alpha.$ That is, there exists a bounded family of operators $(T_z)_{z \in \Sigma_\alpha}$ extending $(T_t)_{t \ge 0}$ and such that $z \mapsto T_z$ is analytic from $\Sigma_\alpha$ into $\mathcal{B} (\mathbf{X}).$ Note that such an extension necessarily satisfies $T_z T_w = T_{z + w}$ for all $z, w \in \Sigma.$

By Proposition \ref{prop:AnalyticSemigroup}, a $C_0$-semigroup $(T_t)_{t \ge 0}$ with the infinitesimal generator $- A$ is a bounded analytic semigroup if and only if $A$ is a sectorial operator of type $\omega$ for some $\omega \in (0, \frac{\pi}{2}).$

\subsection{$H^{\8}$ functional calculus}

Given any $\theta \in (0,\pi),$ we let $H^\infty(\Sigma_\theta)$ be the set of all bounded analytic functions $f:\Sigma_\theta \to \mathbb{C}.$ This is a Banach algebra for the supermum norm
\be
\| f \|_{\8, \theta}: = \sup_{z \in \Sigma_\theta} | f (z) |.
\ee
Then we let $H^\infty_0(\Sigma_\theta)$ be the subalgebra of all $f\in H^\infty(\Sigma_\theta)$ for which there exist two positive numbers $s,c>0$ such that
\beq\label{eq:EstiH_0funct}
| f(z) | \le c \frac{| z |^s}{(1+ | z | )^{2s}},\quad z \in \Sigma_\theta.
\eeq

Now given a sectorial operator $A$ of type $\omega\in(0,\pi)$ on a Banach space $\mathbf{X}$, a number $\theta\in(\omega, \pi ),$ and a function $f\in H^\infty_0(\Sigma_\theta),$ one may define an
operator $f(A)\in \mathcal{B} ( \mathbf{X} )$ as follows. We let $\gamma \in (\omega,\theta)$ be an intermediate angle and consider the oriented contour $\Gamma_\gamma$ defined by
\be
\Gamma_\gamma(t)= \left\{
\begin{split} & -te^{i\gamma},\quad t\in \mathbb{R}_-;\\[0.6mm]
& te^{-i\gamma},\quad t\in \mathbb{R}_+. \end{split} \right.
\ee
In other words, $\Gamma_\gamma$ is the boundary of $\Sigma_\gamma$ oriented counterclockwise. For any $f\in H^\infty_0(\Sigma_\theta),$ we set
\beq\label{eq:f(A)}
f(A)=\frac{1}{2\pi i}\int_{\Gamma_\gamma} f(z)R(z,A)dz.
\eeq
By \eqref{eq:EsitSectorialOper} and \eqref{eq:EstiH_0funct}, it follows that this integral is absolutely convergent. Indeed, \eqref{eq:EstiH_0funct} implies that for any $\gamma \in (0, \theta),$ we have
\be
\int_{\Gamma_{\gamma}} \Big | \frac{f(z)}{z} \Big | | d z| < \8.
\ee
Thus $f(A)$ is a well defined element of $\mathcal{B} (\mathbf{X}).$ It follows from Cauchy's Theorem that the definition of $f(A)$ does not depend on the choice of $\gamma.$ Furthermore, it can be shown that the mapping $f \mapsto f(A)$ is an algebra homomorphism from $H^\infty_0(\Sigma_\theta)$ into $\mathcal{B} ( \mathbf{X}).$

\begin{defi}\label{def:HinftyCalculus}
Let $A$ be a sectorial operator of type $\omega \in(0,\pi)$ on $\mathbf{X}$ and let $\theta \in (\omega, \pi)$. We say that $A$ admits a bounded $H^\infty(\Sigma_\theta)$ functional calculus if
there is a constant $K>0$ such that
\beq\label{eq:HinftyInequa}
\| f(A)\| \leq K \| f \|_{\infty,\theta}, \quad \forall f \in H^\infty_0(\Sigma_\theta).
\eeq
\end{defi}

\begin{remark}\label{rk:HinftyFunctDuality}\rm
Suppose that $\mathbf{X}$ is reflexive and that $A$ is a sectorial operator of type $\omega \in (0, \pi)$ on $\mathbf{X}.$ Then $A^*$ is a sectorial operator of type $\omega$ on $\mathbf{X}^*$ as well. Given $0< \omega < \theta< \pi$ and any $f \in H^{\8} (\Sigma_\theta),$ let us define
\be
\tilde{f} (z) = \overline{f ( \bar{z} )},\quad \forall z \in \Sigma_\theta.
\ee
Then $\tilde{f} \in H^{\8} (\Sigma_\theta)$ and $\| \tilde{f} \|_{\8, \theta} = \| f \|_{\8, \theta}.$ Moreover,
\be
\tilde{f} (A^*) = f(A)^*,\quad \forall f \in H^{\8}_0 (\Sigma_\theta).
\ee
Consequently, $A^*$ admits a bounded $H^\infty(\Sigma_\theta)$ functional calculus whenever $A$ does.
\end{remark}

\begin{remark}\label{rk:HinftyFunctExtension}\rm
For any $\lambda \in \mathbb{C} \setminus \overline{\Sigma}_\theta,$ define $R_\lambda (z) = (\lambda -z)^{-1}.$ Then $R_\lambda \in H^{\8} (\Sigma_\theta).$ Set
\be
\widetilde{H}^{\8}_0 (\Sigma_\theta) = H^{\8}_0 (\Sigma_\theta) \oplus \mathrm{span} \{ 1, R_{-1} \} \subset H^{\8} (\Sigma_\theta).
\ee
This is a subalgebra of $H^{\8} (\Sigma_\theta).$ Now we define
\be
u_A : \widetilde{H}^{\8}_0 (\Sigma_\theta) \mapsto \mathcal{B} (\mathbf{X})
\ee
be the linear mapping such that
\be
u_A (1) = I_{\mathbf{X}},\quad u_A (R_{-1}) = - (1 + A)^{-1},
\ee
and $u_A (f) = f(A)$ for any $f \in H^{\8}_0 (\Sigma_\theta).$ Then, it is easy to check that $u_A$ is an algebra homomorphism and for any $\lambda \in \mathbb{C} \setminus \overline{\Sigma}_\theta,$ we have
\be
R_\lambda \in \widetilde{H}^{\8}_0 (\Sigma_\theta) \quad \text{and}\quad u_A (R_\lambda) = R(\lambda, A).
\ee
$u_A$ is said to be the holomorphic functional calculus of $A$ on $\widetilde{H}^{\8}_0 (\Sigma_\theta).$

Evidently, $A$ admits a bounded $H^\infty(\Sigma_\theta)$ functional calculus if and only if the homomorphism $u_A$ is continuous.
\end{remark}

Let $A$ be a sectorial operator of type $\omega \in (0, \pi)$ and assume that $A$ has dense range. Let $\varphi (z) = z (1+z)^{-2}$ and so $\varphi (A) = A (1 + A)^{-2}.$ Then $\varphi (A)$ is one-one and has dense range (see e.g. \cite[Proposition 2.4]{LeM2007}). Following \cite{M1986, CDMY1996}, we can define an operator $f(A)$ for any $f \in H^{\8} (\Sigma_\theta)$ whenever $\omega < \theta < \pi.$ Indeed, for each $f \in H^{\8} (\Sigma_\theta)$ the product function $f \varphi$ belongs to $H^{\8}_0 (\Sigma_\theta).$ Then using the fact that $\varphi (A)$ is one-one we set
\be
f(A) = \varphi (A)^{-1} (f \varphi) (A)
\ee
with the domain being
\be
\mathrm{D} ( f(A)) = \big \{ x \in \mathbf{X}:\; (f \varphi) (A) (x) \in \mathrm{D} (A) \cap \mathrm{R} (A) \big \}.
\ee
This domain contains $\mathrm{D} (A) \cap \mathrm{R} (A)$ and so is dense in $\mathbf{X}.$ Since $\varphi (A)$ is bounded, $f(A)$ is closed. Therefore, $f(A)$ is bounded if and only if $\mathrm{D} ( f(A)) = \mathbf{X}.$ Note however that $f(A)$ may be unbounded in general.

\begin{thm}\label{th:HinftyCalculus}{\rm (\cite{M1986, CDMY1996})}
Let $0< \omega < \theta < \pi$ and let $A$ be a sectorial operator of type $\omega$ on $\mathbf{X}$ with dense range. Then $f(A)$ is bounded for any $f \in H^{\8} (\Sigma_\theta)$ if and only if $A$ admits a bounded $H^\infty(\Sigma_\theta)$ functional calculus. In that case, we have
\be
\| f(A) \| \le K \| f \|_{\8, \theta},\quad \forall f \in H^{\8} (\Sigma_\theta),
\ee
where the constant $K$ is the one appearing in \eqref{eq:HinftyInequa}.
\end{thm}

\begin{remark}\label{rk:HinftyCalculusReflexiveSpace}\rm
Let $A$ be a sectorial operator on $\mathbf{X}.$ If $\mathbf{X}$ is a reflexive Banach space, then $\mathbf{X}$ has a direct sum decomposition
\be
\mathbf{X} = \mathrm{N} (A) \oplus \overline{\mathrm{R} (A)}
\ee
(see \cite[Theorem 3.8]{CDMY1996}). Then $A$ is one-one if and only if $A$ has dense range. Moreover, the restriction of $A$ to $\overline{\mathrm{R} (A)}$ is a sectorial operator with dense range. Thus changing $\mathbf{X}$ into $\overline{\mathrm{R} (A)},$ or changing $A$ into $A+P$ where $P$ is the projection onto $\mathrm{N} (A)$ with the kernel equals to $\overline{\mathrm{R} (A)},$ it reduces to the case when a sectorial operator has dense range.
\end{remark}

\begin{remark}\label{rk:HinftyCalculusImaginaryPowers}\rm
Given $s \in \mathbb{R},$ let $f_s$ be the analytic function on $\mathbb{C} \setminus (-\8, 0]$ defined by $f_s (z) = z^{\mathrm{i} s}.$ Then $f_s \in H^{\8} (\Sigma_\theta)$ for any $\theta \in (0, \pi)$ with
\be
\| f_s \|_{\8, \theta} = e^{\theta |s|}.
\ee
The imaginary powers of a sectorial operator $A$ with dense range may be defined by letting $A^{\mathrm{i}s}= f_s (A)$ for any $s \in \mathbb{R}.$ In particular, $A^{\mathrm{i}s}$ is bounded for any $s \in \mathbb{R}$ if $A$ admits a bounded $H^\infty(\Sigma_\theta)$ functional calculus for some $\theta \in (0, \pi)$ (see e.g. \cite[Section 5]{CDMY1996}).
\end{remark}

\subsection{Square functions on Hilbert spaces}

Square functions for sectorial operators on Hilbert spaces were introduced by McIntosh in \cite{M1986} and developed further with applications to $H^{\8}$ functional calculus in \cite{MY1989}. We give a brief description of this theory in this subsection.

To this end, we let $\mathbb{H}$ be a Hilbert space throughout the paper. Let $A$ be a sectorial operator of type $\omega \in (0, \pi)$ on $\mathbb{H}.$ We set
\be
H^{\8}_0 (\Sigma_{\omega +}) = \bigcup_{\omega < \theta < \pi} H^{\8}_0 (\Sigma_\theta).
\ee
Then for any $F \in H^{\8}_0 (\Sigma_{\omega +}),$ we set
\beq\label{eq;SquareFunct}
\| x \|_F : = \left ( \int^{\8}_0 \| F (t A ) x \|^2 \frac{d t}{ t} \right )^{\frac{1}{2}}, \quad \forall x \in \mathbb{H}.
\eeq
In the above definition, $F(t A)$ means $F_t (A)$ where $F_t (z) = F (t z).$ By Lebesgue's dominated theorem it is easy to check that for any $x \in \mathbb{H},$ the mapping $t \mapsto F(t A)x$ is continuous and hence $\| x \|_F$ is well defined. However we may have $\| x \|_F = \8$ for some $x.$ We call $\| x \|_F$ a square function associated with $A.$

\begin{thm}\label{th:SquareFunctEquiv}{\rm (McIntosh/Yagi \cite{MY1989})}
Let $\mathbb{H}$ be a Hilbert space. Let $A$ be a sectorial operator of type $\omega \in (0, \pi)$ on $\mathbb{H},$ and suppose that $A$ is one-one. Given $\theta \in (\omega, \pi),$ let $F$ and $G$ be two nonzero functions in $H^{\8}_0 (\Sigma_\theta).$
\begin{enumerate}[{\rm (i)}]

\item There is a constant $K>0$ such that for any $f \in H^{\8} (\Sigma_\theta),$
\be
\left ( \int^{\8}_0 \| f (A) F (t A ) x \|^2 \frac{d t}{ t} \right )^{\frac{1}{2}} \le K \| f \|_{\8, \theta} \| x \|_G, \quad \forall x \in \mathbb{H}.
\ee

\item There is a constant $C>0$ such that
\be
C^{-1} \| x \|_G \le \| x \|_F \le C \| x \|_G,\quad \forall x \in \mathbb{H}.
\ee

\end{enumerate}
\end{thm}

Let $G \in H^{\8}_0 (\Sigma_{\omega +}).$ We denote by $\| \cdot \|^*_G$ the square function for $G$ associated with the adjoint operator $A^*,$ that is,
\be
\| x \|^*_G = \left ( \int^{\8}_0 \| G (t A^* ) x \|^2 \frac{d t}{ t} \right )^{\frac{1}{2}}, \quad \forall x \in \mathbb{H}.
\ee
The following theorem establishes the close connection between $H^{\8}$ functional calculus and square functions on Hilbert spaces.

\begin{thm}\label{th:SquareFunctHinftyCalculus}{\rm (McIntosh \cite{M1986})}
Let $\mathbb{H}$ be a Hilbert space. Let $A$ be a sectorial operator of type $\omega \in (0, \pi)$ on $\mathbb{H},$ and suppose that $A$ is one-one. Given $\theta \in (\omega, \pi),$ the following assertions are equivalent:
\begin{enumerate}[{\rm (i)}]

\item $A$ has a bounded $H^\infty(\Sigma_\theta)$ functional calculus.

\item For some (equivalently, for any) pair $(F, G)$ of nonzero functions in $H^{\8}_0 (\Sigma_{\omega +}),$ there is a constant $K>0$ such that
\be
\| x\|_F \le K \| x \| \quad \text{and}\quad \|x\|^*_G \le K \|x\|
\ee
for all $x \in \mathbb{H}.$

\item For some (equivalently, for any) nonzero function $F \in H^{\8}_0 (\Sigma_{\omega +}),$ there is a constant $C>0$ such that
\be
C^{-1} \| x \| \le \| x \|_F \le C \| x \|,\quad \forall x \in \mathbb{H}.
\ee

\end{enumerate}
\end{thm}

Consequently, for a sectorial operator $A$ of type $\omega \in (0, \pi)$ on a Hilbert space $\mathbb{H},$ if $A$ has a bounded $H^\infty(\Sigma_\theta)$ functional calculus for some $\theta \in (\omega, \pi)$ then it has a bounded $H^\infty(\Sigma_\theta)$ functional calculus for all $\theta \in (\omega, \pi).$ In this case, we simply say that $A$ has a bounded $H^\infty$ functional calculus.

\begin{remark}\label{rk:SquareFunctHinftyCalculus}\rm
$A$ is said to satisfy a square function estimate if for some (equivalently, for any) $F \in H^{\8}_0 (\Sigma_{\omega +}),$ there is a constant $C>0$ such that $\| x \|_F \le C \| x \|$ for all $x \in \mathbb{H}.$ As a consequence of Theorem \ref{th:SquareFunctHinftyCalculus} (and Remark \ref{rk:HinftyCalculusReflexiveSpace}), $A$ has a bounded $H^\infty$ functional calculus if and only if both $A$ and $A^*$ satisfy a square function estimate. Note that an example was given in \cite{LeM2003} of a sectorial operator $A$ which satisfies a square function estimate, but does not have a bounded $H^\infty$ functional calculus.
\end{remark}

Our goal of this paper is to extend Theorems \ref{th:SquareFunctEquiv} and \ref{th:SquareFunctHinftyCalculus} to the case where square functions are replaced by so-called area integral functions defined below.

\section{Area integral functions}\label{AreaFunct}

First of all, we introduce so-called area integral functions associated with sectorial operators on Hilbert spaces.

\begin{defi}\label{df:AreaFunct}
Let $\omega \in (0, \pi)$ and $\theta \in (\omega, \pi).$ Let $A$ be a sectorial operator of type $\omega$ on a Hilbert space $\mathbb{H}.$ Given $0 < \alpha < \frac{\theta - \omega}{2},$ for any $F \in H^\infty_0(\Sigma_{\theta +})$ we define
\beq\label{eq:AreaFunct}
\| x \|_{F, \alpha}: = \left ( \int_{\Sigma_\alpha} \| F (z A ) x \|^2 \frac{d m(z)}{|z|^2} \right )^{\frac{1}{2}}, \quad \forall x \in \mathbb{H},
\eeq
where $d m$ is the Lebesgue measure in $\mathbb{R}^2 \cong \mathbb{C}.$ Here, $F (z A )$ is understood as $F_z (A)$ where $F_z (w) = F (z w)$ for $w \in \Sigma_{\theta-\alpha}.$

We will call $\| x \|_{F, \alpha}$ the area integral function associated with $A.$
\end{defi}

Evidently, for any $z \in \Sigma_\alpha$ one has
\be
F_z \in H^\infty_0(\Sigma_{\theta -\alpha}) \subset H^\infty_0(\Sigma_{\omega +}).
\ee
Also, by Lebesgue's dominated theorem, for any $x \in \mathbb{H}$ the mapping $z \mapsto F_z (A) x$ is continuous from $\Sigma_\alpha$ into $\mathbb{H}.$ Hence, $\| x \|_{F, \alpha}$ is well defined but possibly $\| x \|_{F, \alpha} = \8.$

The corresponding area integral function associated with $A^*$ is defined as
\beq\label{eq:AreaFunctDualOperator}
\| x \|^*_{F, \alpha}: = \left ( \int_{\Sigma_\alpha} \| F (z A^* ) x \|^2 \frac{d m(z)}{|z|^2} \right )^{\frac{1}{2}}, \quad \forall x \in \mathbb{H}.
\eeq

Our main results read as follows.

\begin{thm}\label{th:AreaIntFunctEquiv}
Let $\mathbb{H}$ be a Hilbert space. Let $A$ be a sectorial operator of type $\omega \in (0, \pi)$ on $\mathbb{H},$ and suppose that $A$ is one-one. Given $\theta \in (\omega, \pi)$ and $0 < \alpha, \beta < \frac{\theta - \omega}{2},$ let $F$ and $G$ be two nonzero functions in $H^{\8}_0 (\Sigma_\theta).$
\begin{enumerate}[{\rm (i)}]

\item There is a constant $K>0$ such that for any $f \in H^{\8} (\Sigma_\theta),$
\be
\left ( \int_{\Sigma_\alpha} \| f (A) F (z A ) x \|^2 \frac{d m(z)}{|z|^2} \right )^{\frac{1}{2}} \le K \| f \|_{\8, \theta} \| x \|_{G, \beta}, \quad \forall x \in \mathbb{H}.
\ee

\item There is a constant $C>0$ such that
\be
C^{-1} \| x \|_{G, \beta} \le \| x \|_{F, \alpha} \le C \| x \|_{G, \beta},\quad \forall x \in \mathbb{H}.
\ee

\end{enumerate}
\end{thm}

\begin{thm}\label{th:AreIntFunctHinftyCalculus}
Let $\mathbb{H}$ be a Hilbert space. Let $A$ be a sectorial operator of type $\omega \in (0, \pi)$ on $\mathbb{H},$ and suppose that $A$ is one-one. Given $\theta \in (\omega, \pi)$ and $0 < \alpha < \frac{\theta - \omega}{2},$ the following assertions are equivalent:
\begin{enumerate}[{\rm (i)}]

\item $A$ has a bounded $H^\infty(\Sigma_\theta)$ functional calculus.

\item For some (equivalently, for any) pair $(F, G)$ of nonzero functions in $H^{\8}_0 (\Sigma_{(\omega + \alpha) +}),$ there is a constant $K>0$ such that
\be
\| x \|_{F, \alpha} \le K \| x \| \quad \text{and}\quad \| x \|^*_{G, \alpha} \le K \|x\|
\ee
for all $x \in \mathbb{H}.$

\item For some (equivalently, for any) nonzero function $F \in H^{\8}_0 (\Sigma_{(\omega + \alpha) +}),$ there is a constant $C>0$ such that
\be
C^{-1} \| x \| \le \| x \|_{F, \alpha} \le C \| x \|,\quad \forall x \in \mathbb{H}.
\ee

\end{enumerate}
\end{thm}

\begin{ex}
As similar to the square functions that are used in Stein's book \cite{Stein1970}, area integral functions associated with sectorial operators originate naturally in harmonic analysis. We mention a few classical ones for illustrations. For any $k \ge 1,$ let
\be
G_k = z^k e^{-z},\quad \forall z \in \mathbb{C}.
\ee
Then $G_k \in H^{\8}_0 (\Sigma_{\omega +})$ for any $\omega \in (0, \frac{\pi}{2}).$ Hence, if $A$ is a sectorial operator of type $\omega$ on a Hilbert space for some $\omega \in (0, \frac{\pi}{2}),$ then $G_k$ gives rise area integral functions associated with $A.$ Indeed, if $(T_t)_{t \ge 0}$ is the bounded analytic semigroup generated by $- A,$ we have
\be
G_k (z A)x = z^k A^k e^{-z A} x = (-z)^k \frac{\partial^k}{\partial z^k} (T_z x), \quad z \in \Sigma_{\frac{\pi}{2}- \omega} \; \text{and}\; x \in \mathbb{H}.
\ee
Hence the corresponding area integral function is
\be
\| x \|_{G_k, \alpha} = \left ( \int_{\Sigma_\alpha} |z|^{2(k -1)} \Big \| \frac{\partial^k}{\partial z^k} (T_z x) \Big \|^2 d m (z) \right )^{\frac{1}{2}}, \quad \forall x \in \mathbb{H}
\ee
for any $0 < \alpha < \frac{\pi}{2}- \omega.$ We thus have that
\be
\| x \|_{G_k, \alpha} \thickapprox \| x \|_{G_m, \beta},\quad \forall x \in \mathbb{H}
\ee
for any $k, m \ge 1$ and any $0< \alpha, \beta < \frac{\pi}{2}- \omega.$
\end{ex}

\section{Proofs of main results}\label{pf}

This section is devoted to the proofs of Theorems \ref{th:AreaIntFunctEquiv} and \ref{th:AreIntFunctHinftyCalculus}. Our proofs require two technical variants of the square and area integral functions $\| x \|_F$ and $\| x \|_{F, \alpha}.$

Let $A$ be a sectorial operator of type $\omega \in (0, \pi)$ on $\mathbb{H}.$ Let $\theta \in (\omega, \pi)$ and $0 < \alpha < \frac{\theta - \omega}{2}.$ Given $\epsilon > 0$ and $\delta>0,$ we set for any $F \in H^{\8}_0 (\Sigma_{\theta}),$
\beq\label{eq:SuqareFunctVariant}
G_\epsilon (F)(x): = \left ( \int^{\8}_\epsilon \| F (t A ) x \|^2 \frac{d t}{ t} \right )^{\frac{1}{2}}, \quad \forall x \in \mathbb{H},
\eeq
and
\beq\label{eq:AreaIntFunctVariant}
S_{\alpha, \delta} ( F ) (x) : = \left ( \int_{\Sigma_{\alpha, \delta}} \| F (z A ) x \|^2 \frac{d m(z)}{|z|^2} \right )^{\frac{1}{2}}, \quad \forall x \in \mathbb{H},
\eeq
where $\Sigma_{\alpha, \delta} = \{z \in \mathbb{C}:\; |z| > \delta,\; | \mathrm{Arg} (z)| < \alpha \},$ respectively. Evidently,
\be
\| x \|_F = \lim_{\epsilon \to 0} G_\epsilon (F)(x) \quad \text{and} \quad \| x \|_{F, \alpha} = \lim_{\delta \to 0} S_{\alpha, \delta} ( F ) (x).
\ee

\begin{lem}\label{le:SquareAreaFunct} For any $\epsilon >0,$
\be
G_\epsilon (F)(x) \le \frac{2}{\sqrt{\pi \sin \alpha}} S_{\alpha, \epsilon (1- \sin \alpha)} ( F ) (x),\quad \forall x \in \mathbb{H}.
\ee
Consequently, for every $0 < \alpha < \frac{\theta - \omega}{2}$ we have
\beq\label{eq:Square<AreaFunct}
\| x \|_F \le  \frac{2}{\sqrt{\pi \sin \alpha}}\| x \|_{F, \alpha},\quad \forall x \in \mathbb{H}.
\eeq
\end{lem}

\begin{proof}
Given $t > \epsilon,$ let $D_t$ be the disc in $\mathbb{R}^2 \cong \mathbb{C}$ centered at $(t, 0)$ and tangent to the boundary of $\Gamma_{\alpha, \epsilon (1- \sin \alpha )}.$ Note that the mapping $z \mapsto F(z A) x$ is analytic in $\Sigma_\alpha,$ we have
\be
F (t A) x = \frac{2}{ (\pi  \sin^2 \alpha) t^2}\int_{D_t} F (z A) x d m (z).
\ee
Consequently,
\be
\| F (t A) x \|^2 \le \frac{C_\alpha}{ t^2} \int_{D_t} \| F (z A) x \|^2 d m (z)
\ee
with $C_\alpha = \frac{2}{\pi  \sin^2 \alpha}.$ Then
\be
[ G_\epsilon (F)(x)]^2 \le C_\alpha \int^{\8}_\epsilon \int_{D_t} \| F (z A) x \|^2 \frac{d m (z) d t}{t^3}.
\ee
However, since $\frac{|z|}{1+ \sin \alpha} \le t \le \frac{|z|}{1 - \sin \alpha}$ for any $z \in D_t,$ we have
\be\begin{split}
[ G_\epsilon (F)(x)]^2 & \le C_\alpha \int_{\Sigma_{\alpha, \epsilon (1- \sin \alpha)}} \| F( z A)\|^2 \int^{\frac{|z|}{1- \sin \alpha}}_{\frac{|z|}{1 + \sin \alpha}} \frac{d t}{t^3} d m(z)\\
& = 2 C_\alpha \sin \alpha \int_{\Sigma_{\alpha, \epsilon (1 - \sin \alpha )}} \| F( z A)\|^2 \frac{ d m (z)}{ |z|^2}\\
& = \frac{4}{\pi \sin \alpha} [ S_{\alpha, \epsilon (1- \sin \alpha)} (F)(x)]^2.
\end{split}\ee
This completes the proof.
\end{proof}

\

{\it Proof of Theorem \ref{th:AreaIntFunctEquiv}.}\; Note that the second assertion follows from the first one. Indeed, applying (i) with the constant function $f=1$ yields an estimate $\| x \|_{F, \alpha} \le K \| x \|_{G, \beta}.$ Then (ii) follows by switching the roles of $F$ and $G$ as well as $\alpha$ and $\beta.$

To prove (i), note that
\be
\int_{\Sigma_\alpha} \| f (A) F (z A ) x \|^2 \frac{d m(z)}{|z|^2} = \int^\alpha_{-\alpha} d s \int^{\8}_0 \| f (A) F (t e^{\mathrm{i} s} A ) x \|^2 \frac{d t}{t}.
\ee
By the proof of Theorem \ref{th:SquareFunctEquiv} (i) (see e.g. \cite{ADM1996, MY1989}), there exists a constant $K>0$ such that for any $f \in H^{\8} (\Sigma_\theta)$ and any $s \in (- \alpha, \alpha),$
\be
\left ( \int^{\8}_0 \| f (A) F (t e^{\mathrm{i} s}A ) x \|^2 \frac{d t}{ t} \right )^{\frac{1}{2}} \le K \| f \|_{\8, \theta} \| x \|_G, \quad \forall x \in \mathbb{H}.
\ee
Thus, we deduce that
\beq\label{eq:Area<SquareFunct}
\left ( \int_{\Sigma_\alpha} \| f (A) F (z A ) x \|^2 \frac{d m(z)}{|z|^2} \right )^{\frac{1}{2}} \le \sqrt{2 \alpha} K \| f \|_{\8, \theta} \| x \|_G, \quad \forall x \in \mathbb{H}.
\eeq
By Lemma \ref{le:SquareAreaFunct} we conclude (i).
\hfill $\Box$

\begin{rk}\label{rk:Sqare=AreaFunct}
Taking $f =1$ in \eqref{eq:Area<SquareFunct}, we obtain that
\be
\| x \|_{F, \alpha} \le \sqrt{2 \alpha} K \| x \|_G,\quad \forall x \in \mathbb{H}.
\ee
Combining this inequality with \eqref{eq:Square<AreaFunct} implies that
\beq\label{eq:Square=AreaFunct}
\| x \|_{F, \alpha} \thickapprox \| x \|_G,\quad \forall x \in \mathbb{H}.
\eeq
\end{rk}

\

{\it Proof of Theorem \ref{th:AreIntFunctHinftyCalculus}.}\; This is a straightforward consequence of Theorem \ref{th:SquareFunctHinftyCalculus} and the equivalence relationship \eqref{eq:Square=AreaFunct} between the square and area integral functions.
\hfill $\Box$

\end{document}